\documentclass[a4paper,11pt]{amsart}
\usepackage[margin= 2.5 cm,bottom=19mm]{geometry}
\usepackage{thm-restate}
\usepackage{graphicx}
\usepackage{geometry}
\usepackage{amsthm}
\usepackage{amssymb}
\usepackage{amsmath}
\usepackage{comment}
\usepackage{hyperref}
\usepackage{xcolor}
\usepackage{enumerate}
\usepackage{enumitem}
\usepackage{listings}
\usepackage{complexity}
\usepackage{blkarray}
\usepackage{lmodern}
\usepackage[english]{babel}
\usepackage{accents}
\usepackage{float}
\usepackage{tikz}
\usetikzlibrary{decorations.pathmorphing}

\tikzset{snake it/.style={decorate, decoration=snake}}

\usetikzlibrary{calc}
\usetikzlibrary{decorations.pathreplacing}
\usetikzlibrary{positioning,patterns}
\usetikzlibrary{arrows,shapes,positioning}
\usetikzlibrary{decorations.markings}

\tikzstyle{edge}=[very thick]
\definecolor{bostonuniversityred}{rgb}{0.8, 0.0, 0.0}
\definecolor{arsenic}{rgb}{0.23, 0.27, 0.29}
\tikzstyle{diredge}=[postaction={decorate,decoration={markings,
		mark=at position .95 with {\arrow[scale = 1]{stealth};}}}]

\newcommand{\fitellipsis}[2] 
{\draw [fill=gray]let \p1=(#1), \p2=(#2), \n1={atan2(\y2-\y1,\x2-\x1)}, \n2={veclen(\y2-\y1,\x2-\x1)}
    in ($ (\p1)!0.5!(\p2) $) ellipse [ x radius=\n2/2+0cm, y radius=0.4cm, rotate=\n1];
}

\usepackage[font=small,labelfont=bf]{caption}
\usepackage[font=small,labelfont=normalfont,labelformat=simple]{subcaption}
\graphicspath{{figs/}}

\newtheorem{theorem}{Theorem}[section]
\newtheorem{lemma}[theorem]{Lemma}

\newtheorem{proposition}[theorem]{Proposition}

\theoremstyle{definition}
\newtheorem{definition}[theorem]{Definition}

\author[Martinsson, Steiner]{Anders Martinsson \and Raphael Steiner }
\address[Martinsson]{Department of Computer Science, Institute of Theoretical Computer Science, ETH Z\"{u}rich, Switzerland}
\email{\tt {anders.martinsson}@inf.ethz.ch}
\address[Steiner]{Department of Computer Science, Institute of Theoretical Computer Science, ETH Z\"{u}rich, Switzerland}
\email{\tt {raphaelmario.steiner}@inf.ethz.ch}

\thanks{The second author was supported by an ETH Z\"{u}rich Postdoctoral Fellowship.
}

\date{\today}

\title{Cycle lengths modulo $k$ in expanders}
\begin{document} 
\maketitle

\begin{abstract}
Given a constant $\alpha>0$, an $n$-vertex graph is called an \emph{$\alpha$-expander} if every set $X$ of at most $n/2$ vertices in $G$ has an external neighborhood of size at least $\alpha|X|$. Addressing a question posed by Friedman and Krivelevich in [\emph{Combinatorica}, 41(1), (2021), pp. 53--74], we prove the following result: 

Let $k>1$ be an integer with smallest prime divisor $p$. Then for $\alpha>\frac{1}{p-1}$ every sufficiently large $\alpha$-expanding graph contains cycles of length congruent to any given residue modulo $k$. 

This result is almost best possible, in the following sense: There exists an absolute constant $c>0$ such that for every integer $k$ with smallest prime divisor $p$ and for every positive $\alpha<\frac{c}{p-1}$, there exist arbitrarily large $\alpha$-expanding graphs with no cycles of length $r$ modulo $k$, for some $r \in \{0,\ldots,k-1\}$. 
\end{abstract}

\section{Introduction}

Extremal or structural conditions imposed on graphs which enforce their containment of cycles of certain types and lengths have a long and rich history in graph theory. Most generally, given a class of graphs $\mathcal{G}$ and a subset $L \subseteq \mathbb{N}$ of positive integers, results in this area typically establish conditions which enforce a given graph $G \in \mathcal{G}$ to contain a cycle whose length is in $L$. 

Natural restrictions on the cycle lengths that have been studied intensively are for example of the form  $L=[r,\infty)$, i.e., we search for sufficiently long cycles; $L=\{r\}$, i.e., we search for a cycle of specific length; or $L=\{\ell \in \mathbb{N}|\ell \equiv r \text{ (mod }k)\}$, i.e., we search for a cycle of length congruent to a given residue modulo $k$.

For each of the above types of restrictions and various classes of graphs $\mathcal{G}$, there are numerous results in the literature. It would be impossible to give a just overview of all the individual results here, let us instead refer to~\cite{bondy,erdos,fan,gao,gyarfas,kostochka,liu,mihok,sudakov} for some selected articles on cycle lengths in graphs. Concerning cycle lengths modulo $k$ specifically, let us mention in the following a selected few of noteworthy prior results in more detail, we refer to~\cite{alon4,chen,dean,erdos2,fan,verstraete} for further results on this topic. 

Confirming a conjecture of Burr and Erd\H{o}s, Bollob\'{a}s proved in~\cite{bollobas} that given two fixed integers $r$ and $k$ such that $k$ is odd, every graph of sufficiently large average degree in terms of $r$ and $k$ contains a cycle of length congruent to $r$ modulo $k$. Note that the restriction to odd numbers $k$ in this result is necessary, as dense bipartite graphs do not contain cycles of odd length. 

The bound on the average degree provided by Bollob\'{a}s's proof was not tight, and several improvements have appeared in the literature, culminating in the paper~\cite{sudakov} by Sudakov and Verstra\"{e}te, who determined the minimum average degree required to force a cycle of length $r$ modulo $k$ up to a constant multiplicative error for all possible choice of $r$ and $k$. Thomassen conjectured in~\cite{thomassen} that for every natural numbers $r$ and $k$, every graph of minimum degree at least $k+1$ contains a cycle of length congruent to $2r$ modulo $k$. This conjecture along with related conjectures was recently confirmed in a breakthrough-paper by Gao, Huo, Liu and Ma~\cite{gao}, and previously special cases of the conjecture were proven by Liu and Ma~\cite{liu}. 

While all the previously mentioned results deal with sufficiently dense graphs, also in well-behaved classes of sparse graphs, even in such with bounded maximum degree, one may expect a rich cycle structure modulo $k$. A result confirming this intuition in the case of cubic graphs was recently obtained by Lyngsie and Merker~\cite{lyngsie}, who proved that for every odd integer $k$ there exists an integer $n_0=n_0(k)$ such that every $3$-connected cubic graph on at least $n_0$ vertices contains cycles of lengths congruent to any given residue modulo $k$. Similarly, Thomassen had shown in~\cite{thomassen2} that for odd integers $k$ cubic graphs of sufficiently large girth in terms of $k$ contain cycles of all residues modulo $k$. 

In this paper, we add to this long line of research by establishing almost tight conditions for the existence of cycles of lengths congruent to all residues modulo $k$ in \emph{expanders}, also called \emph{expanding graphs}. Expanders form a rich and well-studied class of graphs with many interesting and useful properties. Roughly speaking, a graph is an expander if the subsets $X$ of vertices in the graph induce neighborhoods $N(X):=\{v \in V\setminus X|v \text{ has a neighbor in }X\}$ whose size grows with $|X|$ reasonably fast. Concretely, one usually requires a linear or at least a polynomial dependency on $|X|$. Because of their numerous desirable properties and applications, constructing expanders with certain features and understanding their general structural properties is and has been an active field of research, we refer to~\cite{hoory,krivel} for informative survey articles on expanding graphs.

In this paper, we adopt the following fairly common variant of expansion called \emph{$\alpha$-expansion}.

\begin{definition}[cf.~\cite{krivel}]
Let $\alpha>0$. Then an $n$-vertex graph $G$ is called $\alpha$-expanding (or also an $\alpha$-expander) if every subset $X$ of vertices of $G$ with $|X| \le \frac{n}{2}$ satisfies $|N_G(X)| \ge \alpha |X|$. 
\end{definition}

Note that the above definition is only sensible for $\alpha \le 1$, since no set $X$ of exactly $\frac{n}{2}$ vertices in an $n$-vertex graph can have an external neighborhood of size more than $1\cdot |X|$.

This notion has been previously utilized in several places, compare for instance the recent survey article~\cite{krivel}, or~\cite{alon3,friedman} and Chapter~9 in~\cite{alonspencer}. The notion of $\alpha$-expansion can be seen as a common denominator of several other previously considered concepts of expander graphs which is still weak enough to encompass several important graph classes, yet strong enough to guarantee interesting structural properties and substructures, see the discussion in~\cite{krivel}. Many interesting classes of graphs fall under the umbrella of $\alpha$-expanders for some constant $\alpha>0$: They can be obtained from supercritical random graphs~\cite{krivel2}, are contained in graphs without small separators~\cite{krivel}, and also the famous $(n,d,\lambda)$-graphs, these are the $n$-vertex $d$-regular graphs with second largest eigenvalue in absolute value $\lambda$, are $\alpha$-expanders for $\alpha=\frac{d-\lambda}{2d}$ (cf.~Corollary 9.9.2 in~\cite{alonspencer}).

Our investigations in this paper are directly motivated by the paper by Friedman and Krivelevich~\cite{friedman}, in which the authors obtain several novel results on the cycle spectrum (i.e., the set of cycle lengths) in $\alpha$-expanding graphs. For instance, they show that in sufficiently large $\alpha$-expanding graphs, the cycle spectrum has only ``gaps'' of constant size. More concretely (cf. Lemma~\ref{lemma:cyclelength} in Section~\ref{sec:proofs}), they prove that for every given integer $\ell \in [O(\ln n), \Omega(n)]$, an $n$-vertex $\alpha$-expanding graph contains a cycle whose length approximates $\ell$ up to a constant error (which depends only on $\alpha$), this also implies that $\alpha$-expanding graphs contain cycles of linearly (in $n$) different lengths. 

The modular arithmetic of cycle lengths has not yet been intensively researched in expanding graphs. Consequently, at the end of their paper, Friedman and Krivelevich raised the open problem to find conditions guaranteeing the containment of cycles of length $r$ modulo $k$ in $\alpha$-expanding graphs, and explicitly asked the question whether cycles of length $0$ modulo $k$ may be guaranteed in $\alpha$-expanding graphs.

The latter question was recently taken up in the paper by Alon and Krivelevich~\cite{alon}, in which they provided a positive answer as follows. In one of their main results, Alon and Krivelevich proved that there exists an absolute constant $C>0$ such that every graph containing a $K_f$-minor for $f \ge Ck \ln k$ contains a cycle of length $0$ mod $k$. This result was subsequently slightly improved by M\'{e}sz\'{a}ros and the second author in~\cite{tamas}, showing that requiring $f \ge Ck$ for some constant $C$ suffices. 
As noted by Alon and Krivelevich in~\cite{alon}, it follows from a result by Kawarabayashi and Reed~\cite{kawarabayashi} that every $n$-vertex $\alpha$-expanding graph contains a clique $K_f$ of order $f \ge c(\alpha) \sqrt{n}$ as a minor, where $c(\alpha)$ is a positive constant depending only on $\alpha$. Combining this with their result on divisible cycles in complete minors, they concluded that for every fixed integer $k$, sufficiently large $\alpha$-expanders contain cycles of length $0$ modulo $k$. Combining with the improved quantitative bound from~\cite{tamas}, we may state the following version of the result by Alon and Krivelevich.

\begin{theorem}[cf.~\cite{alon,tamas}]\label{thm:0modk}
Let $k>0$ be an integer and $\alpha>0$. Then there exists an integer $n_0=n_0(\alpha,k)=O_\alpha(k^2)$ such that every $\alpha$-expanding graph on at least $n_0$ vertices contains a cycle of length congruent to $0$ modulo $k$.
\end{theorem}

While the above result largely resolves the question of Friedman and Krivelevich from~\cite{friedman} in the special case $r=0$, it leaves open all the remaining cases when $r \not\equiv0 \text{ (mod }k)$. In fact, the method used by Alon and Krivelevich to first obtain a large complete minor and then to find in it a cycle with the desired length modulo $k$ does not seem feasible for cycles of length non-zero modulo $k$: For every fixed $k \in \mathbb{N}$ there are graphs containing arbitrarily large clique minors, but in which all cycles are of length divisible by $k$\footnote{A possible construction looks as follows: Take a complete graph $K_f$ for some large integer $f$, and subdivide each of its edges into a path of length $k$. This graph contains $K_f$ as a minor, but all its cycles are of length divisible by $k$.}. 

In this paper, we take up Friedman and Krivelevich's question for non-zero residues modulo $k$, and find a complete answer for the question under which circumstances we may guarantee cycles of \emph{all possible} residues $r$ modulo $k$ in large $\alpha$-expanding graphs. 

The following is our main result. 

\begin{theorem}\label{thm:modk}
Let $k>1$ be an integer, and let $p$ be the smallest prime divisor of $k$. Then for every $\alpha>\frac{1}{p-1}$ there exists $n_0=n_0(\alpha,k) \in \mathbb{N}$ such that for each $r \in \{0,1,\ldots,k-1\}$ every $\alpha$-expanding graph on at least $n_0$ vertices contains a cycle of length congruent to $r$ modulo $k$. 
\end{theorem}

We will present the proof of Theorem~\ref{thm:modk} in Section~\ref{sec:proofs}. It is rather different from the complete minor-approach from~\cite{alon}, and instead follows a more direct plan, in which we carefully construct a closed chain of several paths and cycles with special conditions, through which we then may route cycles of all possible lengths modulo $k$.

Note that in contrast to Theorem~\ref{thm:0modk}, in Theorem~\ref{thm:modk} we need the additional assumption that $\alpha$ is not too small ($\alpha>\frac{1}{p-1}$, where $p$ is the smallest prime factor of $k$). While this might seem restrictive, it is in fact close to optimal: The following proposition shows that if we were to choose $\alpha$ just by a multiplicative constant smaller, then we would no longer be able to guarantee cycles of all residues modulo $k$, even in arbitrarily large $\alpha$-expanders. 

\begin{proposition}\label{prop:nomodk}
There exists an absolute constant $c>0$ such that the following holds. 

For every integer $k>1$ with smallest prime divisor $p$, and every $\alpha \in (0,\frac{c}{p-1})$ there exist arbitrarily large $\alpha$-expanding graphs in which all cycles are of length divisible by $p$. 
\end{proposition}
In particular, the graphs from the above proposition contain no cycles of length congruent to $r$ modulo $k$ for every $r \in \{1,\ldots,k-1\}$ not divisible by $p$ (for instance, for $r=1$). 

\medskip

\textbf{Structure of the paper.} In Section~\ref{sec:not} we give a brief summary of notation used in the paper. In the remainder, we present the proofs of Theorem~\ref{thm:modk} and Proposition~\ref{prop:nomodk} in Section~\ref{sec:proofs}. We end the paper in Section~\ref{sec:op} with a few open problems. 

\section{Notation}\label{sec:not}
Given a graph $G$, we denote by $V(G)$ its vertex-set, by $E(G)$ its edge-set, and by $v(G)$ and $e(G)$ the respective sizes. Given a subset $X \subseteq V(G)$, we denote by $G[X]$ the subgraph of $G$ induced by the vertices in $X$, and we denote $G-X:=G[V(G)\setminus X]$. We use the notation $N_G(X):=\{v \in V(G)\setminus X|v \text{ has a neighbor in }X\}$ for the \emph{external neighborhood} of $X$ in $G$. The diameter of a connected graph is the maximum pairwise distance of two vertices in the graph. Given a path $P$ or a cycle $C$ in a graph, we denote by $V(P), V(C)$ the sets of their vertices and by $\ell(P), \ell(C)$ their lengths (i.e., their number of edges). Given an integer $k \in \mathbb{N}$, we denote by $\mathbb{Z}_k$ the cyclic additive group of order $k$, whose elements we sometimes represent by the numbers $\{0,\ldots,k-1\}$. Given an integer $r \in \mathbb{Z}$, we denote by $r \text{ mod }k \in \mathbb{Z}_k \simeq \{0,1,\ldots,k-1\}$ its congruent residue modulo $k$. Given two subsets $A_1, A_2$ of $\mathbb{Z}_k$, we denote by $A_1+A_2:=\{a_1+a_2|a_1 \in A_1, a_2 \in A_2\}$ their sumset. 

\section{Proofs}\label{sec:proofs}
In this section, we prove Theorem~\ref{thm:modk} and Proposition~\ref{prop:nomodk}. We prepare the proof of Theorem~\ref{thm:modk} with several auxiliary statements, starting with the following basic but crucial ``removal''-lemma. It roughly speaking states that if we delete a sufficiently small proportion of vertices from an $\alpha$-expanding graph, we can retain most of its expansion properties after ``cleaning'' the graph from small separators by deleting some additional vertices. We note that statements similar (but not identical) to Lemma~\ref{lemma:delete} appeared in the literature before, for instance in~\cite{krivel}. 

\begin{lemma}\label{lemma:delete}
Let $0<\beta<\alpha$, and let $G$ be an $n$-vertex $\alpha$-expander. Then for every non-empty set of vertices $X \subseteq V(G)$ with $|X| \le \frac{(\alpha-\beta)^2}{4(\alpha-\beta)+2}n$ there exists $Y \subseteq V(G)\setminus X$ of size $|Y| < \frac{|X|}{\alpha-\beta}$ such that $G-(X \cup Y)$ is a $\beta$-expander. 
\end{lemma}
\begin{proof}
If $G-X$ is a $\beta$-expander, then the statement holds with $Y:=\emptyset$. So assume from now on that $G-X$ is not a $\beta$-expander. This implies that $$\mathcal{Z}:=\left\{Z \subseteq V(G) \setminus X \Bigm| |N_{G-X}(Z)|<\beta |Z| \text{ and } |Z| \le \frac{n}{2}\right\} \neq \emptyset.$$ Let $Y$ be chosen as a member of the above set collection of maximum size. Then we have $|N_{G-X}(Y)|<\beta|Y|$ and $|N_{G}(Y)| \ge \alpha |Y|$, which implies
$$\alpha |Y| \le |N_G(Y)| \le |N_{G-X}(Y)| + |X|<\beta|Y|+|X|,$$ and hence $|Y|<\frac{|X|}{\alpha-\beta}$. It remains to be shown that for every $U \subseteq V(G-(X \cup Y))$ with $|U| \le \frac{v(G-(X\cup Y))}{2}$ we have $|N_{G-(X \cup Y)}(U)| \ge \beta |U|$. So let a set $U \subseteq V(G-(X \cup Y))$ with $|U| \le \frac{v(G-(X\cup Y))}{2}<\frac{n}{2}$ be given arbitrarily, and let us show that $|N_{G-(X \cup Y)}(U)| \ge \beta|U|$. Since the latter holds trivially if $U=\emptyset$, in the following we may assume $U \neq \emptyset$. Next, we distinguish between two possible cases, depending on whether $|Y \cup U| \le \frac{n}{2}$, or $|Y \cup U|>\frac{n}{2}$. 

Suppose first that $|Y \cup U| \le \frac{n}{2}$. Then we know by the maximality of $Y$ within $\mathcal{Z}$ that $|N_{G-X}(Y \cup U)| \ge \beta |Y \cup U|$. This yields
$$\beta|Y \cup U| \le |N_{G-X}(Y \cup U)| \le |N_{G-X}(Y) \cup N_{G-(X \cup Y)}(U)|$$ $$\le |N_{G-X}(Y)|+|N_{G-(X \cup Y)}(U)|<\beta|Y|+|N_{G-(X \cup Y)}(U)|.$$ Subtracting $\beta|Y|$ from both sides of the above inequality now yields $|N_{G-(X \cup Y)}(U)| \ge \beta|U|$, as desired. 
Second, suppose that $|Y \cup U|>\frac{n}{2}$, i.e., $|U|>\frac{n}{2}-|Y|$. We can now estimate as follows:

$$|N_{G-(X \cup Y)}(U)| \ge |N_{G}(U)|-|X|-|Y| \ge \alpha|U|-|X|-|Y|>\alpha|U|-\left(1+\frac{1}{\alpha-\beta}\right)|X|$$ $$=\beta|U|+(\alpha-\beta)|U|-\left(1+\frac{1}{\alpha-\beta}\right)|X| > \beta|U|+(\alpha-\beta)\frac{n}{2}-(\alpha-\beta)|Y|-\left(1+\frac{1}{\alpha-\beta}\right)|X|$$
$$>\beta|U|+(\alpha-\beta)\frac{n}{2}-\left(2+\frac{1}{\alpha-\beta}\right)|X| \ge \beta|U|,$$
where for the last inequality we used that $|X| \le \frac{(\alpha-\beta)^2}{4(\alpha-\beta)+2}n$ by assumption. 

We have thus shown that indeed, $G-(X \cup Y)$ is a $\beta$-expander, and this concludes the proof of the lemma.
\end{proof}

In the following two lemmas, we collect two additional ingredients for our proof of Theorem~\ref{thm:modk}. The first lemma is a simple (Folklore-)fact about expanding graphs which can be found explicitly in~\cite{krivel}, while the second is a powerful result proved by Friedman and Krivelevich in~\cite{friedman}.

\begin{lemma}[cf.~Corollary~3.2 in~\cite{krivel}]\label{lemma:diameter}
Let $\alpha>0$. Then every $\alpha$-expanding $n$-vertex graph has diameter at most $\left\lceil\frac{2(\ln n -1)}{\ln(1+\alpha)}\right\rceil$.
\end{lemma}
\begin{lemma}[cf.~Theorem~1 in~\cite{friedman}]\label{lemma:cyclelength}
Let $\alpha>0$. Then there exist constants $A=A(\alpha), N=N(\alpha) \in \mathbb{N}$ and $a_1=a_1(\alpha), a_2=a_2(\alpha)>0$ such that every $n$-vertex $\alpha$-expanding graph $G$ with $n \ge N$ satisfies the following:

For every integer $\ell \in [a_1\ln n, a_2n]$, there is a cycle in $G$ whose length is between $\ell$ and $\ell+A$. 
\end{lemma}

We are now ready for the proof of Theorem~\ref{thm:modk}.

\begin{proof}[Proof of Theorem~\ref{thm:modk}]
Let $k>1$ be an integer, let $p$ denote the smallest prime divisor of $k$, and let $\alpha>\frac{1}{p-1}$ be a constant. W.l.o.g. we may assume that $k$ is odd, since for even $k$ we have $p=2$ and hence $\alpha>1$, which means that no non-trivial $\alpha$-expanding graphs exist, and the claim of the Theorem holds vacuously. 

Let $\varepsilon \in \left(0,\frac{1}{6}\right)$ be chosen small enough such that $\alpha-3\varepsilon>\frac{1}{p-1}$, denote $\alpha':=\alpha-\varepsilon$, and let $A=A(\varepsilon), N=N(\varepsilon) \in \mathbb{N}$ and $a_1=a_1(\varepsilon), a_2=a_2(\varepsilon)>0$ be constants as given by Lemma~\ref{lemma:cyclelength}. Fix a constant $D \in \mathbb{N}$ such that $D \ge \frac{2k}{a_1\ln(1+\varepsilon)}+\frac{Ak+2}{a_1}$ and $D \ge \frac{2p}{\varepsilon}+1$. Finally, let $n_0$ be defined as the smallest positive integer such that all of the following inequalities hold for all integers $n \ge n_0$:

\begin{align}
    & n \ge 4N \\
    & \lceil a_1 \ln n \rceil \cdot D^{k+1} \le \left\lfloor a_2\frac{n}{4}\right\rfloor \\
    & k\cdot \frac{2\ln(n)}{\ln(1+\varepsilon)}+\lceil a_1\ln n\rceil \sum_{j=1}^{k}{D^{j+1}}+Ak
\le \left\lfloor\frac{\varepsilon^2}{4\varepsilon+2} n \right\rfloor
\end{align}

Starting the proof, let us consider any given $\alpha$-expanding graph $G$ on $n \ge n_0$ vertices. Our aim is to show that $G$ contains cycles with lengths congruent to any given residue modulo $k$.

Let $R \subseteq V(G)$ be a set such that $G[R]$ is connected and $|R|=\left\lfloor \frac{\varepsilon^2}{4\varepsilon+2} n \right\rfloor$ (such a set exists, since $G$ is an $\alpha$-expander and hence connected, thus it contains a spanning tree $T$, and we can pick $R$ as the vertex-set of a properly sized subtree of $T$ by successively removing leafs). In the remainder of the proof, $R$ will be called the \emph{reservoir}. At the end of our construction, the reservoir will be used for closing paths into cycles of given length modulo $k$ in $G$. 

The key step in our proof is to construct a chain of cycles in $G-R$, each two consecutive of which are connected by a short path, in such a way that the ends of this cycle chain can be connected by paths of every possible length modulo $k$ through the cycle chain. For technical reasons, in our construction of this chain it is important that the lengths of the cycles are exponentially increasing. Formally, we establish the following central claim.

\medskip

\textbf{Claim~1.} There exists a number $t \in \{1,\ldots,k\}$, vertex-disjoint cycles $C_1,C_2,\ldots, C_{t}$ in $G-R$, vertex-disjoint paths $P_1,P_2,\ldots, P_{t}$ in $G$, and pairwise distinct vertices $$v_0 \in R; u_1,v_1 \in V(C_1); u_2,v_2 \in V(C_2); \ldots; u_{t} \in V(C_{t}),$$ such that all the following properties are satisfied:
\begin{itemize}
    \item For every $i \in \{1,\ldots,t\}$, we have $\ell_i \le \ell(C_i) \le \ell_i+A$, where  $\ell_i:=\lceil a_1\ln n \rceil\cdot D^{i+1}$.
    \item For every $i \in \{1,\ldots,t\}$ the path $P_i$ has endpoints $v_{i-1}$ and $u_{i}$ and is internally vertex-disjoint from all cycles $C_1,\ldots,C_{t}$ as well as from $R$.
    \item For every $i \in \{1,\ldots,t\}$, we have $|V(P_i)\setminus \{v_{i-1},u_i\}| \le \frac{2\ln(n)}{\ln(1+\varepsilon)}$. 
    \item For each $i \in \{1,\ldots,t-1\}$, let $Q_i^{1}$ and $Q_i^{2}$ denote the two internally disjoint paths in $C_i$ which connect $u_i$ and $v_i$, and let $z_i:=(\ell(Q_i^{1})-\ell(Q_i^{2})) \text{ mod }k$ denote their length difference modulo $k$. Then for every $i \in \{1,2,\ldots,t\}$ the set 
    $$B_i:=\left\{\sum_{j \in J}{z_j}\middle| J \subseteq \{1,\ldots,i-1\}\right\} \subseteq \mathbb{Z}_k$$ (summation as in the cyclic group $(\mathbb{Z}_k,+)$) contains at least $i$ elements.
    \item $B_t=\mathbb{Z}_k$. 
\end{itemize}

\begin{proof}[Proof of Claim~1.]
We prove the claim by successively constructing the vertices, cycles and paths such that they satisfy the required properties. 

We start by applying Lemma~\ref{lemma:delete} to $G$, where $X:=R$ and $\beta:=\alpha'$. This yields the existence of a set $S \subseteq V(G) \setminus R$ such that $|S| \le \frac{|R|}{\alpha-\alpha'}=\varepsilon^{-1}|R| \le \frac{\varepsilon}{4\varepsilon+2}n$ and such that $G':=G-(R \cup S)$ is an $\alpha'$-expander. Note that $\alpha'>\varepsilon$ implies $G'$ also is an $\varepsilon$-expander. Let $n':=v(G')$. By our choice of $n_0$, we have $n' \ge n-|R|-|S| \ge n-\frac{\varepsilon^2+\varepsilon}{4\varepsilon+2}n>\frac{3}{4}n \ge \frac{3}{4}n_0 \ge N$. Furthermore, $n \ge n_0$ also implies $\lceil a_1 \ln n'\rceil \le \lceil a_1 \ln n\rceil \cdot D^2 \le  \lfloor a_2 n' \rfloor$. Hence, we may invoke Lemma~\ref{lemma:cyclelength} with parameter $\ell_1=\lceil a_1 \ln n \rceil \cdot D^2$ to find a cycle $C_1$ in $G'$ whose length satisfies $\ell_1 \le \ell(C_1) \le \ell_1+A$. Let $P_1$ be a shortest path in $G$ which starts in a vertex in $R$ and ends in a vertex in $C_1$. Then clearly $P_1$ intersects each of $R$ and $V(C_1)$ in exactly one vertex, let $v_0 \in R$ and $u_1 \in V(C_1)$ be such that $V(P_1) \cap R=\{v_0\}$ and $V(P_1) \cap V(C_1)=\{u_1\}$. By Lemma~\ref{lemma:diameter}, we further know that the number of internal vertices of $P_1$ is bounded by $|V(P_1)\setminus\{v_0,u_1\}|=\ell(P_1)-1<\frac{2(\ln n -1)}{\ln(1+\alpha)} \le \frac{2\ln(n)}{\ln(1+\alpha)} \le \frac{2\ln(n)}{\ln(1+\varepsilon)}$. Finally, since we are only summing over the empty set $J$ in the definition of $B_1$, we have $B_1=\{0\} \subseteq \mathbb{Z}_k$. With this, it is readily checked that all four items of the claim with $i=1$ are satisfied. 

In the following construction, while not required by the statement in the claim, for inductive purposes we will maintain the following additional invariant while constructing the sequence of cycles and paths: 

\medskip 

For every constructed cycle $C_i$ in the sequence we have $V(C) \subseteq V(G')$, and for every constructed path $P_i$ in the sequence with $i \ge 2$ we have $V(P_i) \subseteq V(G')$. 

\medskip

Moving on, suppose that for some $i \in \{1,\ldots,k-1\}$ we have already constructed vertices $v_0,u_1,v_1,\ldots,v_{i-1},u_i$, paths $P_1,\ldots,P_{i}$ and cycles $C_1,\ldots,C_{i}$ such that the four items in the claim and the additional invariant above are satisfied for all indices $1,2,\ldots,i$. If $B_i=\mathbb{Z}_k$, then we can simply put $t:=i$, which will verify the claim. So, instead, assume in the following that $B_i \neq \mathbb{Z}_k$, and let us go about constructing the vertices $v_i,u_{i+1}$, the path $P_{i+1}$, and the cycle $C_{i+1}$. 

Consider the last already constructed cycle $C_i$. Fix a circular direction of traversal around $C_i$, and with respect to this direction, for every two distinct vertices $x,y \in V(C_i)$, denote by $C_i[x,y]$ the subpath of $C_i$ starting at $x$ and ending at $y$, following the circular orientation of the cycle from $x$ to $y$. For every vertex $x \in V(C_i)$, denote by $\text{diff}(x) \in \mathbb{Z}_k$ the quantity $(\ell(C_i[u_i,x])-\ell(C_i[x,u_i])) \text{ mod }k$, which measures by how much the two paths connecting $u_i$ and $x$ differ in length modulo $k$. 

Let $Z:=\text{stab}(B_i) \subseteq \mathbb{Z}_k$ denote the stabilizer of $B_i \subseteq \mathbb{Z}_k$ in the cyclic group $(\mathbb{Z}_k,+)$, i.e., $z \in Z$ if and only if $B_i+z=\{b+z|b \in B_i\}=B_i$. Note that $Z$ forms a subgroup of $(\mathbb{Z}_k,+)$, and that since $0 \in B_i$ by definition, we have $Z \subseteq B_i$.

Finally, let us define a set $K \subseteq V(C_t)$ of ``bad vertices'' on $C_t$ by 
$$K:=\{x \in V(C_i)\setminus\{u_i\}|\text{diff}(x) \in Z\}.$$
We now claim that $|K| \le \frac{|V(C_i)|}{p}+1$. First of all, note that $Z \subseteq B_i \subsetneq \mathbb{Z}_k$ by our assumption above. This means that $Z$ is a proper subgroup of $(\mathbb{Z}_k,+)$. Let $d>1$ be the index of $Z$ in $(\mathbb{Z}_k,+)$. Then $d$ must be a divisor of $k$, which in particular means that $p \le d$. Let now $x \in K$ be given arbitrarily. Note that we have the following equalities:

$$2\ell(C_i[u_i,x])-\ell(C_i) =\ell(C_i[u_i,x])-\ell(C_i[x,u_i]) \equiv \text{diff}(x) \text{ (}\text{mod }k).$$

Noting that $2^{-1}Z=Z$, where $2^{-1}$ denotes the multiplicative inverse of $2$ in $\mathbb{Z}_k$ (recall that $k$ is odd), we can see that $\ell(C_i[u_i,x]) \text{ mod }k \in 2^{-1}\ell(C_i)+2^{-1}Z=2^{-1}\ell(C_i)+Z$ for every $x \in K$. Since $d$, seen as an element of $\mathbb{Z}_k$, generates the subgroup $Z$, the above in turn implies that $\ell(C_i[u_i,x]) \equiv c \text{ (}\text{mod }d)$ for some constant $c \in \{0,1,\ldots,k-1\}$ and every $x \in K$. This implies that the numbers $\ell(C_i[u_i,x]), x \in K$ are of pairwise distance at least $d$ and hence, there can be no more than $\frac{|V(C_i)|}{d}+1 \le \frac{|V(C_i)|}{p}+1$ of them. Since the vertices $x \in K$ are uniquely determined by $\ell(C_i[u_i,x])$, this proves that also  $|K|\le \frac{|V(C_i)|}{p}+1$, as desired.

Let us now apply Lemma~\ref{lemma:delete} to the $\alpha'$-expander $G'$, with $\beta:=\varepsilon$ and $$X_i:=(V(P_1) \cap V(G')) \cup \bigcup_{j=2}^{i}{(V(P_j)\setminus\{v_{j-1},u_j\})} \cup \bigcup_{j=1}^{i}{V(C_j)}.$$ Note that this is possible, since
$$|X_i| \le \sum_{j=1}^{i}{|V(P_j)\setminus\{v_{j-1},u_j\}|}+\sum_{j=1}^{i}{|V(C_j)|}$$
$$\le i\cdot \frac{2\ln(n)}{\ln(1+\varepsilon)}+\lceil a_1\ln n\rceil \sum_{j=1}^{i}{D^{j+1}}+Ai$$
$$\le k\cdot \frac{2\ln(n)}{\ln(1+\varepsilon)}+\lceil a_1\ln n\rceil \sum_{j=1}^{k}{D^{j+1}}+Ak$$
$$\le \frac{\varepsilon^2}{4\varepsilon+2} n,$$
$$\le \frac{(\alpha'-\varepsilon)^2}{4(\alpha'-\varepsilon)+2} n,$$
where moving from the third to the fourth line we used that $n \ge n_0$. In particular, the above shows that $|X_i| \le \frac{n}{4}$ (note that $\alpha'-\varepsilon \le \alpha \le 1$). 

From the lemma we obtain a set $Y_i \subseteq V(G')\setminus X_i$ such that $|Y_i|<\frac{|X_i|}{\alpha'-\varepsilon} \le \frac{\alpha'-\varepsilon}{4(\alpha'-\varepsilon)+2}n<\frac{1}{4}n$ and such that $G'-(X_i \cup Y_i)$ is an $\varepsilon$-expander. Since $v(G'-(X_i \cup Y_i)) \ge n'-2\cdot \frac{1}{4}n \ge \frac{3}{4}n-\frac{n}{2}=\frac{n}{4} \ge N$ by our choice of $n_0$, we can now apply Lemma~\ref{lemma:cyclelength} to the $\varepsilon$-expander $G'-(X_i \cup Y_i)$. Again using that $n \ge n_0$, we have $$\lceil a_1\cdot \ln v(G'-(X_i \cup Y_i)) \rceil \le \lceil a_1 \ln n \rceil \le \ell_{i+1}=\lceil a_1 \ln n \rceil \cdot D^{i+2} \le \left\lfloor a_2\frac{n}{4}\right\rfloor \le  \lfloor a_2 \cdot v(G'-(X_i \cup Y_i))\rfloor,$$ so that the lemma guarantees the existence of a cycle $C_{i+1}$ in $G'-(X_i \cup Y_i)$ with length $\ell(C_{i+1}) \in [\ell_{i+1},\ell_{i+1}+A]$. Note that by definition of $G'$ and $X_{i}$, the cycle $C_{i+1}$ is disjoint from $R$, the paths $P_1,\ldots,P_i$, and the cycles $C_1,\ldots,C_i$. 

Next, we want to construct the short path $P_{i+1}$ which connects $C_i$ to $C_{i+1}$. In order to do so, we again apply Lemma~\ref{lemma:delete} to $G'$, but this time with a slightly different choice for the set $X$: Let 
$$X_i':=(V(P_1) \cap V(G')) \cup \bigcup_{j=2}^{i}{(V(P_j)\setminus\{v_{j-1},u_j\})} \cup \bigcup_{j=1}^{i-1}{V(C_j)} \cup K \cup \{u_i\}.$$
Clearly, we have $X_i' \subseteq X_i$, and thus from the above $|X_i'| \le |X_i| \le \frac{(\alpha'-\varepsilon)^2}{4(\alpha'-\varepsilon)+2} n.$ Applying Lemma~\ref{lemma:delete} with parameters $\alpha',\varepsilon$ to $G'$ and $X_i'$ now yields a subset $Y_i' \subseteq V(G) \setminus X_i'$ such that $G'-(X_i' \cup Y_i')$ is $\varepsilon$-expanding and $|Y_i'|<\frac{|X_i'|}{\alpha'-\varepsilon} < \frac{|X_i'|}{\frac{1}{p-1}+\varepsilon}$. We now claim that $V(C_{i})\setminus (X_i' \cup Y_i') \neq \emptyset$ and  $V(C_{i+1})\setminus (X_i' \cup Y_i') \neq \emptyset$. Since $\ell_i \le |V(C_i)|$ and $\ell_i \le \ell_{i+1} \le |V(C_{i+1})|$, this will be shown once we have established that $|X_i' \cup Y_i'|=|X_i'|+|Y_i'| < \ell_i$. 

We start estimating $|X_i'|$ as follows:
$$|X_i'| \le \sum_{j=1}^{i}{|V(P_j)\setminus\{v_{j-1},u_j\}|}+\sum_{j=1}^{i-1}{|V(C_j)|}+|K|+1$$
$$\le i\cdot \frac{2\ln(n)}{\ln(1+\varepsilon)}+\lceil a_1\ln n\rceil \sum_{j=1}^{i-1}{D^{j+1}}+A(i-1)+\left(\frac{|V(C_i)|}{p}+1\right)+1$$
$$< i\cdot \frac{2\ln(n)}{\ln(1+\varepsilon)}+\frac{\lceil a_1\ln n\rceil \cdot D^{i+1}}{D-1}+A(i-1)+2+\frac{A}{p}+\frac{\ell_i}{p}$$
$$\le \left(k\cdot \frac{2\ln(n)}{\ln(1+\varepsilon)}+Ak+2\right)+\left(\frac{1}{p}+\frac{1}{D-1}\right)\ell_i$$
$$=\left(\left(k\cdot \frac{2\ln(n)}{\ln(1+\varepsilon)}+Ak+2\right)-\frac{1}{D-1}\ell_i\right)+\left(\frac{1}{p}+\frac{2}{D-1}\right)\ell_i$$

We further have by our choice of $D$, that $$\frac{1}{D-1}\ell_i \ge \frac{1}{D-1}\cdot \lceil a_1 \ln n\rceil \cdot D^2> a_1 \ln (n) \cdot D \ge k\cdot \frac{2\ln(n)}{\ln(1+\varepsilon)}+Ak+2,$$ as well as $\frac{1}{p}+\frac{2}{D-1} \le \frac{1+\varepsilon}{p}$ (since $D \ge \frac{2k}{a_1\ln(1+\varepsilon)}+\frac{Ak+2}{a_1}$ and $D \ge \frac{2p}{\varepsilon}+1$). All in all, this yields the following final estimate on the size of $X_i'$:

$$|X_i'| < \frac{1+\varepsilon}{p}\ell_i.$$

Using the facts that $p \ge 3$ ($k$ is odd) and $\varepsilon<\frac{1}{6}$, one can verify that $1+\frac{1}{\frac{1}{p-1}+\varepsilon}<(1-\varepsilon)p$. Putting our estimates together, we now obtain:

$$|X_i'|+|Y_i'|<\left(1+\frac{1}{\frac{1}{p-1}+\varepsilon}\right)|X_i'| < (1-\varepsilon)p|X_i'|<(1-\varepsilon)p\frac{1+\varepsilon}{p}\ell_i=(1-\varepsilon^2)\ell_i<\ell_i,$$ as desired. This shows that indeed, $V(C_i) \setminus (X_i' \cup Y_i')$ and $V(C_{i+1}) \setminus (X_i' \cup Y_i')$ are non-empty, as claimed. Since $G'-(X_i' \cup Y_i')$ is an $\varepsilon$-expander, it is connected. Let us now finally pick the path $P_{i+1}$ as a shortest path connecting $V(C_i) \setminus (X_i' \cup Y_i')$ to $V(C_{i+1}) \setminus (X_i' \cup Y_i')$ in $G'-(X_i' \cup Y_i')$. Then clearly $P_{i+1}$ intersects both $C_i$ and $C_{i+1}$ in exactly one vertex. Define $v_i$ and $u_{i+1}$ as the unique vertices such that $V(P_{i+1}) \cap V(C_i)=\{v_i\}$ and $V(P_{i+1}) \cap V(C_{i+1})=\{u_{i+1}\}$. Note that by definition of $G'$ and $X_i'$, we know that $P_{i+1}$ is internally vertex-disjoint from $C_1,\ldots,C_{i+1}$ and $P_1,\ldots,P_{i}$. Since $P_{i+1}$ is a shortest path in the $\varepsilon$-expander $G'$, by Lemma~\ref{lemma:delete} we further know that $\ell(P_{i+1}) \le \left\lceil \frac{2\ln(n)}{\ln(1+\varepsilon)} \right\rceil$, which implies that $|V(P_{i+1})\setminus\{v_i,u_{i+1}\}|=\ell(P_{i+1})-1<\frac{2\ln(n)}{\ln(1+\varepsilon)}$. These observations show that our newly constructed $C_{i+1},P_{i+1},v_i,u_{i+1}$ satisfy the first three items in the claim (for index $i+1$), and it remains to verify that the subset $B_{i+1}$ of $\mathbb{Z}_k$, defined as in the claim, contains at least $i+1$ elements. Towards a contradiction, suppose that $|B_{i+1}| \le i$. Since the definition implies $B_i \subseteq B_{i+1}$, and since we have $|B_i| \ge i$ (as we already established the claim with index $i$), it follows that $B_{i+1}=B_i$. We claim that this implies that $z_{i} \in \text{stab}(B_i)=Z$. Towards a contradiction, suppose that $B_i+z_i \neq B_i$. This implies the existence of an element $b \in B_i$ such that $b+z_i \notin B_i$ or $b-z_i \notin B_i$. Let $q \in \mathbb{N}$ denote the order of the element $z_i$ in the cyclic group $(\mathbb{Z}_k,+)$, and consider the following sequence of elements:
$b,b+z_i,b+2z_i,\ldots,b+(q-1)z_i=b-z_i$. From the above we have that the first element of the above sequence is contained in $B_i$, while at least one element of the sequence is not contained in $B_i$. This implies the existence of some $j \in \mathbb{N}$ such that $b+(j-1)z_i \in B_i$ but $b+jz_i \notin B_i$. The definition of $B_{i+1}$ then however immediately yields $b+jz_i=(b+(j-1)z_i)+z_i \in B_i+z_i \subseteq B_{i+1}=B_i$, a contradiction. This shows that indeed, we must have $z_{i} \in Z$. Noting that by definition (possibly after renaming $Q_i^1$ and $Q_i^2$), we have $z_{i}=\text{diff}(v_i)$, this directly implies that $v_i \in K \cup \{u_i\}$.

However, since $v_i \notin X_i' \supseteq K \cup \{u_i\}$ we know that the latter is not the case. This is a contradiction to the above assumption that $B_{i+1}$ contains at most $i$ elements, and shows that indeed $|B_{i+1}| \ge i+1$, as required by the fourth item in the claim. 

Our above argumentation shows that we may construct the vertices, paths and cycles satisfying the requirements of the claim successively, up until the first time in the iterative construction the last constructed set $B_i$ satisfies $B_i=\mathbb{Z}_k$. In this case, we put $t:=i$ and verify the claim in this way. From the invariant $|B_i| \ge i$ which we maintain during the construction process, we see that this must happen after at most $k$ steps, which finally implies that the value of $t$ we find in this way indeed satisfies $t \le k$. 
This concludes the proof of Claim~1.
\end{proof}

With Claim~1 at hand, it is now easy to conclude the proof of the theorem. Essentially, we will now close paths contained in our cycle-chain constructed in Claim~1 into cycles, by establishing a connection from $C_t$ to $R$ which is internally vertex-disjoint from all the cycles and paths $P_1,P_2,\ldots,P_t,C_1,\ldots,C_t$. 

\medskip

\textbf{Claim~2.} There exists a path $P$ in $G$ with endpoints $x \in V(C_t)$, $y \in R$ such that $P$ is internally vertex-disjoint from the paths $P_1,\ldots,P_t$ and disjoint from the cycles $C_1,\ldots,C_t$, and such that $V(P) \cap V(C_t)=\{x\}, V(P) \cap V(R)=\{y\}$. 

\begin{proof}[Proof of Claim 2.]
Let $X:=\bigcup_{i=1}^{t}{\left(V(P_i)\setminus\{v_{i-1},u_i\}\right)} \cup \bigcup_{i=1}^{t-1}{V(C_i)}$. We claim that there exists a path in $G-X$ starting in a vertex in $V(C_t)$ and ending in a vertex of $R$. Once this fact is established, it will imply the claim by picking $P$ as a shortest such path in $G-X$.  

Towards a contradiction, suppose such a path does not exist. This implies that there exists a partition of the vertex-set of $G-X$ into two disjoint parts $S \supseteq V(C_t)$ and $T \supseteq R$ such that no edge in $G-X$ connects $S$ and $T$. This in particular implies that $N_G(S), N_G(T) \subseteq X$. Note that at least one of $S$ and $T$ has size at most $\frac{n}{2}$. Since $G$ is an $\alpha$-expander, this implies that $|N_G(S)| \ge \alpha|S|$ or $|N_G(T)| \ge \alpha|T|$. In each case, it follows that $$|X| \ge \alpha\min\{|S|,|T|\} \ge \alpha\min\{|V(C_t)|,|R|\} \ge \alpha\min\left\{\lceil a_1\ln(n)\rceil \cdot D^{t+1},\left\lfloor \frac{\varepsilon^2}{4\varepsilon+2}n\right\rfloor\right\}.$$
On the other hand, estimating the different parts of $X$ yields: 
$$|X| \le \sum_{j=1}^{t}{|V(P_j)\setminus\{v_{j-1},u_j\}|}+\sum_{j=1}^{t-1}{|V(C_j)|}$$
$$\le t\cdot \frac{2\ln(n)}{\ln(1+\varepsilon)}+\lceil a_1\ln n\rceil \sum_{j=1}^{t-1}{D^{j+1}}+At$$
$$< k\cdot \frac{2\ln(n)}{\ln(1+\varepsilon)}+Ak+\lceil a_1\ln n\rceil \frac{D^{t+1}}{D-1}$$
$$\le 2\lceil a_1\ln n \rceil \frac{D^{t+1}}{D-1}$$
$$\le \alpha \lceil a_1\ln n \rceil D^{t+1},$$
where moving from the third to the fourth line we used that $D \ge \frac{2k}{a_1\ln(1+\varepsilon)}+\frac{Ak+2}{a_1}$, while moving from the fourth to the fifth line we used that $D\ge \frac{2p}{\varepsilon}+1>\frac{2}{\alpha}+1$.

Further, our assumptions on $n_0$ and $n \ge n_0$ imply that that $\lceil a_1\ln n \rceil D^{t+1} \le \lceil a_1\ln n \rceil D^{k+1} \le \left\lfloor\frac{\varepsilon^2}{4\varepsilon+2}\right\rfloor n$.
Together, these estimates yield that $$|X|<\alpha \lceil a_1\ln n \rceil D^{t+1}=\alpha\min\left\{\lceil a_1\ln(n)\rceil \cdot D^{t+1},\left\lfloor \frac{\varepsilon^2}{4\varepsilon+2}n\right\rfloor\right\},$$ contradicting the above. This contradiction proves Claim~2.
\end{proof}

Let us fix some final notation before concluding the proof of the theorem. Recall from Claim~1 that for every $i \in \{1,\ldots,t-1\}$ we denote by $Q_{i}^1$ and $Q_i^2$ the two internally disjoint subpaths of $C_i$ connecting $u_i$ to $v_i$. We also let $Q_t$ denote a subpath of the cycle $C_t$ connecting $u_t$ to $x$ (the latter path may consist only of a single vertex if $u_t=x$). Further, let $Q_0$ denote a path in $G[R]$ connecting $y$ to $v_0$ (note that such a path exists, since $R$ was initially chosen such that it induces a connected subgraph of $G$). Again, $Q_0$ may consist of a single vertex only. Finally, for every subset $J \subseteq \{1,\ldots,t-1\}$ let us denote by $C(J)$ the cycle in $G$ which is obtained as the union of the following paths: 

$$(P_1,\ldots,P_t), (Q_i^1, i \in J); (Q_i^2, i \notin J); Q_t, P, Q_0$$
We claim that for every residue $r \in \{0,\ldots,k-1\}$ there exists a subset $J$ of $\{1,\ldots,t-1\}$ such that $\ell(C(J)) \equiv r \text{ (mod }k)$, which will prove that $G$ contains cycles of all parities modulo $k$, and hence the theorem statement. 

Indeed, it is easily seen by definition of the cycles that $$\ell(C(J))=\ell(C(\emptyset))+\sum_{j \in J}{(\ell(Q_i^{1})-\ell(Q_i^{2}))}$$ for every $J \subseteq \{1,\ldots,t-1\}$. As $B_t=\mathbb{Z}_k$ by Claim~1, we see that $\sum_{j \in J}{(\ell(Q_i^{1})-\ell(Q_i^{2}))}$, taken modulo $k$, attains all possible residues modulo $k$ as $J$ ranges over the subsets of $\{1,\ldots,t-1\}$. This, however, means that also $\ell(C(J)) \text{ mod }k$ takes on all possible residues modulo $k$ as $J$ varies on the subsets of $\{1,\ldots,t-1\}$. Finally, this proves the above assertion and concludes the proof of the theorem.
\end{proof}

We conclude this section with the proof of Proposition~\ref{prop:nomodk}.

\begin{proof}[Proof of Proposition~\ref{prop:nomodk}]
It is known (cf.~\cite{krivel}, Proposition 4.2) that for any fixed positive integer $\Delta$ there exists a constant $c_\Delta>0$ such that for every integer $\ell$ and given any $\alpha$-expander $G$ of maximum degree $\Delta$, the graph $G^{(\ell)}$ obtained from $G$ by subdividing every edge of $G$ exactly $\ell$ times is a $c_\Delta\cdot\frac{\alpha}{\ell}$-expander.

Let $(G_n)_{n=1}^\infty$ be an infinite sequence of $\alpha_0$-expanding cubic graphs, where $\alpha_0$ is an absolute constant, and $v(G_n) \rightarrow \infty$ for $n \rightarrow \infty$. The existence of such a constant $\alpha_0$ and such a sequence can be seen in multiple ways, one possible source of examples are the random cubic graphs $G_{n,3}$, we refer to the article~\cite{bollobas2} for an analysis of their expansion properties. Further, in the recent article~\cite{alon2} the interested reader may find a method for explicitely constructing bounded degree expander graphs.

Let now $k>1$ be a given integer with smallest prime divisor $p$. Consider the sequence of graphs $\left(G_n^{(p-1)}\right)_{n=1}^{\infty}$. By the above all graphs in this seqeunce are $\frac{c_3\alpha_0}{p-1}$-expanding, and graphs in the sequence can be arbitrarily large. However, since every edge in $G_n$ is turned into a path of length $p$ in $G_n^{(p-1)}$, we can see that every cycle in $G_n^{(p-1)}$ is of length divisible by $p$. This implies the assertion of the proposition with the constant $c:=c_3\alpha_0$. 
\end{proof}

\section{Conclusion}\label{sec:op}
In this note, we have demonstrated conditions which guarantee the existence of cycles of all lengths modulo $k$ in large $\alpha$-expanding graphs. The following questions remain open. 

\begin{itemize}
    \item Looking into our proof of Theorem~\ref{thm:modk}, one can see that the smallest value of $n_0$ for which we may guarantee the conclusion of the theorem grows at least as fast as $\Omega(k^k)$. Is it possible to obtain the same conclusion as in Theorem~\ref{thm:modk} with a more moderate assumption on the size of the graphs?
    \item Can we drop the assumption $\alpha=\Omega\left(\frac{1}{p-1}\right)$ in the statement of Theorem~\ref{thm:modk} if we assume additional conditions on the graph $G$, such as sufficiently good connectivity? Is it true that for every $\alpha>0$ and odd $k \in \mathbb{N}$ every sufficiently large $3$-connected $\alpha$-expander contains cycles of all lengths modulo $k$?
\end{itemize}

\end{document}